\documentclass[11pt]{amsart}
\usepackage{amsopn}
\usepackage{amssymb, amscd}
\usepackage{multirow}
\usepackage{graphicx, graphics, epsfig}
\usepackage{faktor} 
\usepackage{enumerate}

\usepackage{tikz}

\usepackage{hyperref}

\newcommand{\nc}{\newcommand}

\nc{\fg}{\mathfrak{f} } \nc{\vg}{\mathfrak{v} } \nc{\wg}{\mathfrak{w} }
\nc{\zg}{\mathfrak{z} } \nc{\ngo}{\mathfrak{n} } \nc{\kg}{\mathfrak{k} }
\nc{\mg}{\mathfrak{m} } \nc{\bg}{\mathfrak{b} } \nc{\ggo}{\mathfrak{g} }
\nc{\ggob}{\overline{\mathfrak{g}} } \nc{\sog}{\mathfrak{so} }
\nc{\sug}{\mathfrak{su} } \nc{\spg}{\mathfrak{sp} } \nc{\slg}{\mathfrak{sl} }
\nc{\glg}{\mathfrak{gl} } \nc{\cg}{\mathfrak{c} } \nc{\rg}{\mathfrak{r} }
\nc{\hg}{\mathfrak{h} } \nc{\tg}{\mathfrak{t} } \nc{\ug}{\mathfrak{u} }
\nc{\dg}{\mathfrak{d} } \nc{\ag}{\mathfrak{a} } \nc{\pg}{\mathfrak{p} }
\nc{\sg}{\mathfrak{s} } \nc{\affg}{\mathfrak{aff} } \nc{\qg}{\mathfrak{q} }

\nc{\pca}{\mathcal{P}} \nc{\nca}{\mathcal{N}} \nc{\lca}{\mathcal{L}}
\nc{\oca}{\mathcal{O}} \nc{\mca}{\mathcal{M}} \nc{\tca}{\mathcal{T}}
\nc{\aca}{\mathcal{A}} \nc{\cca}{\mathcal{C}} \nc{\gca}{\mathcal{G}}
\nc{\sca}{\mathcal{S}} \nc{\hca}{\mathcal{H}} \nc{\bca}{\mathcal{B}}
\nc{\dca}{\mathcal{D}} \nc{\val}{\operatorname{val}}

\nc{\vp}{\varphi} \nc{\ddt}{\tfrac{d}{dt}} \nc{\dds}{\frac{d}{ds}}
\nc{\dpar}{\tfrac{\partial}{\partial t}} \nc{\im}{\mathrm{i}}

\nc{\SO}{\mathrm{SO}} \nc{\Spe}{\mathrm{Sp}} \nc{\Sl}{\mathrm{SL}}
\nc{\SU}{\mathrm{SU}} \nc{\Or}{\mathrm{O}} \nc{\U}{\mathrm{U}} \nc{\Gl}{\mathrm{GL}}
\nc{\Se}{\mathrm{S}} \nc{\Cl}{\mathrm{Cl}} \nc{\Spein}{\mathrm{Spin}}
\nc{\Pin}{\mathrm{Pin}} \nc{\G}{\mathrm{GL}_n} \nc{\g}{\mathfrak{gl}_n}

\nc{\RR}{{\Bbb R}} \nc{\HH}{{\Bbb H}} \nc{\CC}{{\Bbb C}} \nc{\ZZ}{{\Bbb Z}}
\nc{\FF}{{\Bbb F}} \nc{\NN}{{\Bbb N}} \nc{\QQ}{{\Bbb Q}} \nc{\PP}{{\Bbb P}} \nc{\OO}{{\Bbb O}}

\nc{\vs}{\vspace{.2cm}} \nc{\vsp}{\vspace{1cm}} \nc{\ip}{\langle\cdot,\cdot\rangle}
\nc{\ipp}{(\cdot,\cdot)} \nc{\la}{\langle} \nc{\ra}{\rangle} \nc{\unm}{\tfrac{1}{2}} \nc{\unt}{\tfrac{1}{3}}
\nc{\unc}{\tfrac{1}{4}} \nc{\uns}{\tfrac{1}{6}} \nc{\no}{ \noindent}
\nc{\lam}{\Lambda^2(\RR^n)^*\otimes\RR^n} \nc{\tangz}{{\rm T}^{\rm Zar}}
\nc{\nor}{{\sf n}}  \nc{\mum}{/\!\!/} \nc{\kir}{/\!\!/\!\!/}
\nc{\Ri}{\tfrac{4\Ric_{\mu}}{||\mu||^2}} \nc{\ds}{\displaystyle}
\nc{\ben}{\begin{enumerate}} \nc{\een}{\end{enumerate}} \nc{\f}{\frac}
\nc{\lb}{[\cdot,\cdot]} \nc{\isn}{\tfrac{1}{||v||^2}}
\nc{\gkp}{(\ggo=\kg\oplus\pg,\ip)} \nc{\ukh}{(\ug=\kg\oplus\hg,\ip)}
\nc{\tgkp}{(\tilde{\ggo}=\kg\oplus\pg,\ip)}
\nc{\wt}{\widetilde} \nc{\Mm}{M}
\nc{\iop}{\mathtt{i}} \nc{\jop}{\mathtt{j}}

\nc{\Hess}{\operatorname{Hess}} \nc{\ad}{\operatorname{ad}}
\nc{\Ad}{\operatorname{Ad}} \nc{\rank}{\operatorname{rank}}
\nc{\Irr}{\operatorname{Irr}} \nc{\End}{\operatorname{End}}
\nc{\Aut}{\operatorname{Aut}} \nc{\Inn}{\operatorname{Inn}}
\nc{\Der}{\operatorname{Der}} \nc{\Ker}{\operatorname{Ker}}
\nc{\Iso}{\operatorname{Iso}} \nc{\Diff}{\operatorname{Diff}}
\nc{\Lie}{\operatorname{L}} \nc{\tr}{\operatorname{tr}} \nc{\dif}{\operatorname{d}}
\nc{\sen}{\operatorname{sen}} \nc{\modu}{\operatorname{mod}}
\nc{\CRic}{\operatorname{PP}} \nc{\Cric}{\operatorname{P}} \nc{\Ricci}{\operatorname{Ric}}
\nc{\sym}{\operatorname{sym}} \nc{\herm}{\operatorname{herm}} \nc{\symac}{\operatorname{sym^{ac}}}
\nc{\symc}{\operatorname{sym^{c}}} \nc{\scalar}{\operatorname{scal}}
\nc{\grad}{\operatorname{grad}} \nc{\ricci}{\operatorname{Rc}}
\nc{\Nor}{\operatorname{Norm}}  \nc{\ricc}{\operatorname{Rc^{c}}}
\nc{\Ricc}{\operatorname{Ric^{c}}} \nc{\ricac}{\operatorname{Rc^{ac}}}
\nc{\Ricac}{\operatorname{Ric^{ac}}} \nc{\Riem}{\operatorname{Rm}} \nc{\Sec}{\operatorname{Sec}}
\nc{\riccig}{\operatorname{ric^{\gamma}}} \nc{\Rin}{\operatorname{M}}
\nc{\Le}{\operatorname{L}} \nc{\tang}{\operatorname{T}}
\nc{\level}{\operatorname{level}} \nc{\rad}{\operatorname{r}}
\nc{\abel}{\operatorname{ab}} \nc{\CH}{\operatorname{CH}} \nc{\Cone}{{\mathcal C}} \nc{\CCone}{\operatorname{CC}} \nc{\CP}{{\mathcal P}}
\nc{\mcc}{\operatorname{mcc}} \nc{\Adj}{\operatorname{Adj}}
\nc{\Order}{\operatorname{O}}  \nc{\inj}{\operatorname{inj}} \nc{\proy}{\operatorname{pr}}
\nc{\vol}{\operatorname{vol}} \nc{\Diag}{\operatorname{Diag}} \nc{\Diagg}{\operatorname{Diag}}
\nc{\Spec}{\operatorname{Spec}} \nc{\Ima}{\operatorname{Im}} \nc{\Rea}{\operatorname{Re}}
\nc{\spann}{\operatorname{span}} \nc{\Aff}{\operatorname{Aff}} \nc{\mm}{\operatorname{m}} 
\nc{\Crit}{\operatorname{Crit}} \nc{\En}{\operatorname{E}}

\theoremstyle{plain}
\newtheorem{theorem}{Theorem}[section]
\newtheorem{proposition}[theorem]{Proposition}
\newtheorem{corollary}[theorem]{Corollary}

\theoremstyle{definition}
\newtheorem{definition}[theorem]{Definition}

\theoremstyle{remark}
\newtheorem{remark}[theorem]{Remark}

\newtheorem{example}[theorem]{Example}

\title{The search for solitons on homogeneous spaces}

\author{Jorge Lauret} 

\address{Universidad Nacional de C\'ordoba and CIEM, CONICET (Argentina)}
\email{lauret@famaf.unc.edu.ar}

\thanks{The author gratefully acknowledges support from  FONCyT (ANPCyT) and SECyT (UNC)}

\begin{document}

\maketitle

\begin{abstract}
The concept of soliton, in its most general version, allows us to find canonical or distinguished elements on any set provided with an equivalence relation and an `optimal' tangent direction at each point.  We study in this paper solitons on homogeneous spaces, which have consolidated its role as a quite useful tool to find soliton geometric structures in Riemannian, pseudo-Riemannian, complex, symplectic and $G_2$ geometries.  
\end{abstract}

\tableofcontents

\section{Introduction}\label{intro2}

We aim to continue here the study of geometric flows and their solitons on homogeneous spaces, with emphasis in a unified approach regardless of the type of geometry as done in \cite{minimal, SCF, BF, solitons}.  The present paper is intended to be an invitation to explore the class of homogeneous spaces beyond Lie groups, with regard to the existence, uniqueness and structure of solitons in pseudo-Riemannian, Hermitian, almost-K\"ahler and exceptional holonomy geometries.    

A {\it homogeneous space} is a quotient $M=G/K$ of a Lie group $G$ by a closed subgroup $K\subset G$.  They are differentiable manifolds and their algebraic provenance promises a rich interplay.  The presentation of a differentiable manifold as a homogeneous space is far from being unique, and attached to each $M=G/K$ we have the subgroup $\Aut(G/K)\subset\Diff(M)$ of {\it equivariant diffeomorphisms} defined by automorphisms of $G$ taking $K$ onto $K$.  On the other hand, the natural left action of $G$ on $M$ determines $G$-actions by pull-back on any space of tensor fields on $M$, giving rise to the concept of $G$-{\it invariant} geometric structure on $M$.  Two $G$-invariant geometric structures (or tensor fields) are said to be {\it equivariantly equivalent} when they belong to the same $\Aut(G/K)$-orbit.   

Given a homogeneous space $M=G/K$ and some kind of geometry, we ask ourselves the following question borrowed from the first page of Besse's book \cite{Bss}:

\begin{itemize}
\item[ ] Are there any best (or nicest, or distinguished) $G$-invariant structures on $M$? 
\end{itemize}
The meaning of these adjectives are, of course, part of the problem.  Classical possibilities to approach this question include special curvature properties, critical points of geometric functionals, parallelism relative to some connection, etc.  The property is expected to be weak enough to allow existence results for large classes of homogeneous spaces, but also sufficiently strong to produce uniqueness or finiteness results up to equivariant equivalence.
 
In this paper, the following dynamical approach is considered.  Let $\Gamma$ be a space of $G$-invariant geometric structures on a fixed homogeneous space $M=G/K$.  Assume that at each $\gamma\in\Gamma$, there is a $G$-invariant optimal or preferred direction $q(\gamma)\in T_\gamma\Gamma$, viewed as a `direction of improvement' in some sense.  The corresponding geometric flow on $\Gamma$, defined by the ODE 
\begin{equation}\label{flow-LG-intro}
\ddt\gamma(t)=q(\gamma(t)), \qquad \gamma(0)=\gamma,
\end{equation}
is therefore supposed to `improve' the structures while they are flowing.  In this light, it is natural to consider a structure $\gamma$ distinguished in the case when the solution to \eqref{flow-LG-intro} is given by 
$$
\gamma(t)=c(t)f(t)^*\gamma, \qquad\mbox{for some}\quad c(t)\in\RR^*,  \quad f(t)\in\Aut(G/K), 
$$
i.e.\ $\gamma(t)$ is self-similar relative to equivariant equivalence.  In that case, $(G/K,\gamma)$ is called a {\it semi-algebraic soliton} in the literature.  The condition is equivalent to the following nice blend of geometric and algebraic aspects of the homogeneous structure $(G/K,\gamma)$:   
\begin{equation}\label{sas-intro}
q(\gamma)=c\gamma+\theta(D_\pg)\gamma, \qquad \mbox{for some} \quad c\in\RR, \quad 
D=\left[\begin{matrix} \ast&\ast\\ 0&D_\pg\end{matrix}\right]\in\Der(\ggo),
\end{equation}
where $\theta$ is the usual $\glg(\pg)$-representation on tensors.  Here we are using a {\it reductive decomposition} $\ggo=\kg\oplus\pg$ (i.e.\ $\Ad(K)\pg\subset\pg$) for the homogeneous space $G/K$ and the fact that any $G$-invariant tensor field on $M=G/K$ can be identified with an $\Ad(K)$-invariant tensor on $T_oM\equiv\pg$.   

If the derivation $D$ in \eqref{sas-intro} satisfies in addition that $D_\pg$ is orthogonal to the stabilizer $\glg(\pg)_\gamma$, then $(G/K,\gamma)$ is said to be an {\it algebraic soliton}.  They also distinguish themselves by being precisely the fixed points of the bracket flow as well as by their `diagonal' evolution (see below).  The {\it bracket flow} is a dynamical system defined on the variety of Lie algebras which is equivalent to the geometric flow \eqref{flow-LG-intro} in a very precise sense.  

The concepts of algebraic and semi-algebraic solitons have a long and fruitful history in the Ricci flow case.  More recently, they have also provided a quite useful tool to study the existence problem of homogeneous soliton geometric structures in many different geometries, including pseudo-Riemannian, complex, symplectic and $G_2$ (see Section \ref{alg-sol-sec} for references).  

After surveying the topics of general solitons, homogeneous geometric structures and semi-algebraic solitons in Sections \ref{intro}, \ref{hom-sec} and \ref{alg-sol-sec}, respectively, the following new results are obtained.  

\begin{enumerate}[{\small $\bullet$}]
\item Provided that a certain mild condition on the flow holds, we prove that any soliton $(M,\gamma)$ in the general sense (i.e.\ $q(\gamma)=c\gamma+\lca_X\gamma$ for some $c\in\RR$ and $X\in\mathfrak{X}(M)$), which is homogeneous, is a semi-algebraic soliton when presented as a homogeneous space $(G/K,\gamma)$ with $G=\Aut(M,\gamma)$ (see Section \ref{evol-sec}).  

\item In Section \ref{mba-sec}, we first give overviews on the moving-bracket approach (to study homogeneous geometric structures) and on the bracket flow.  Secondly, we compute the evolution of semi-algebraic solitons relative to the bracket flow.

\item We geometrically characterize algebraic solitons among homogeneous solitons as those developing a `simultaneously diagonalizable' solution $\gamma(t)$ to the ODE \eqref{flow-LG-intro} (see Section \ref{fd-sec}).  
\end{enumerate}

\vspace{.5cm} \noindent {\it Acknowledgements.}   The author gratefully acknowledges Valeria Guti\'errez and Marina Nicolini for helpful comments during the preparation of this paper.

\section{Solitons in differential geometry}\label{intro}

From a very general point of view, the necessary ingredients to define a {\it soliton} are just the following (see \cite{solitons}): 

\begin{enumerate}[{\small $\bullet$}]
\item A set $\Gamma$ with a notion of tangent space $T_\gamma\Gamma$ at each $\gamma\in\Gamma$.  

\item An equivalence relation $\simeq$ on $\Gamma$.  

\item An optimal or preferred direction at each point,  $q(\gamma)\in T_\gamma\Gamma$, viewed as a `direction of improvement' in some sense.  
\end{enumerate}
In that case, $\gamma\in\Gamma$ is called a {\it soliton} if 
\begin{equation}\label{sol-def}
q(\gamma)\in T_\gamma[\gamma],    
\end{equation}
where $[\gamma]$ is the equivalence class of $\gamma$, that is, $\gamma$ is in a way `nice' enough that it does not need to be `improved'.  

In the context of differential geometry, we consider a space $\Gamma$ of geometric structures on a fixed differentiable manifold $M$ and identify $\Gamma$ with a subset of the vector space $\tca^{r,s}M$ of all tensor fields of some type $(r,s)$ (or tuples of tensors).  As usual, the equivalence relation is scaling and pulling-back by diffeomorphisms, i.e.\ the equivalence class of $\gamma\in\Gamma$ is given by 
$$
[\gamma]=\left\{ ch^*\gamma:c\in\RR,\; h\in\Diff(M)\right\}\cap\Gamma. 
$$   
Typically, a preferred direction, 
$$
\gamma\mapsto q(\gamma)\in T_\gamma\Gamma\subset\tca^{r,s}M,
$$ 
is given by a curvature tensor associated to some affine connection attached to $\gamma$, or the gradient of a natural geometric functional, or the Hodge-Laplacian on differential forms, etc.  

Once the space $\Gamma$ and the preferred direction $q$ have been specified, it follows from \eqref{sol-def} that $\gamma\in\Gamma$ is a soliton if and only if  
\begin{equation}\label{sol-def-DG}
q(\gamma)=c\gamma+\lca_X\gamma, \qquad  \mbox{for some} \quad c\in\RR^*, \quad X\in\mathfrak{X}(M), 
\end{equation}
where $\lca_X$ denotes Lie derivative with respect to a vector field $X$ of $M$ (recall that $T_\gamma(\Diff(M)\cdot\gamma)=\lca_{\mathfrak{X}(M)}\gamma$).  In the case when $q$ is {\it diffeomorphism equivariant} (i.e.\ $q(f^*\gamma)=f^*q(\gamma)$ for any $f\in\Diff(M)$ and $\gamma\in\Gamma$), which will be assumed in this paper, any geometric structure in $\Gamma$ which is equivalent to a soliton is also a soliton.   

\begin{example}\label{RicS} ({\it Ricci solitons})   
The space $\Gamma$ of all Riemannian metrics on $M$ is open in $\sca^2M$, the vector space of all symmetric $2$-tensors, and a natural preferred direction is 
$$
q(g):=-2\Ricci_g\in T_g\Gamma=\sca^2M\subset\tca^{2,0}M,
$$ 
where $\Ricci_g$ is the Ricci tensor of the metric $g\in\Gamma$.  This gives rise to the well-known Ricci solitons.   
\end{example} 

\begin{example}
If a complex manifold $(M,J)$ is fixed and a space $\Gamma$ of hermitian metrics on $(M,J)$ are to be considered, then the equivalence is determined by the action of the group $\Aut(M,J)$ of bi-holomorphic maps rather than $\Diff(M)$, so $X$ has to be a holomorphic field.  In the symplectic case, the group of symplectomorphisms $\Aut(M,\omega)$ of a fixed symplectic manifold $(M,\omega)$ must be considered.  
\end{example}

Concerning the associated geometric flow,
\begin{equation}\label{flow-DG}
\dpar\gamma(t)=q(\gamma(t)), \qquad \gamma(0)=\gamma,  
\end{equation}
one easily obtains that $\gamma$ is a soliton if and only if 
\begin{equation}\label{sss}
\gamma(t)=c(t)f(t)^*\gamma, \qquad  \mbox{for some} \quad c(t)\in\RR, \quad f(t)\in\Diff(M),
\end{equation}
i.e.\ $\gamma(t)$ is a {\it self-similar} solution.  

\begin{example} ({\it Ricci flow})
The evolution equation determined by Example \ref{RicS} is precisely the famous Ricci flow $\dpar g(t)=-2\Ricci_{g(t)}$ introduced in the 80s by Hamilton and used as a primary tool by Perelman to prove the Poincar\'e and Geometrization conjectures.    
\end{example}

Note that solitons are not necessarily fixed points of the flow (i.e.\ zeroes of $q$).  Nevertheless, they are not either improved by the flow and may attract or stop other solutions in its way to a fixed point.  Thus their existence is undesirable if one is hoping to use the flow to find zeroes of $q$ in $\Gamma$; on the other hand, the existence of solitons is great news for the search of canonical or distinguished structures in $\Gamma$ beyond zeroes of $q$.        

\begin{remark}
A natural preferred direction $q$ may or may not produce a flow, as the existence of solutions to the PDE \eqref{flow-DG} is not guaranteed.  So possibly, a study of solitons can be worked out without any reference to a flow, as in the case of Ricci solitons in pseudo-Riemannian geometry (see \cite{ClvGrc}).  
\end{remark}

Let us assume for the rest of the paper that the scaling behavior of the preferred direction $q$ is given by \begin{equation}\label{alfa}
q(c\gamma)=c^\alpha\gamma, \qquad \forall c\in\RR^*, \quad \gamma\in\Gamma, 
\end{equation}
for some fixed $\alpha<1$.  In that case, the scaling in \eqref{sss} is given by 
$$
c(t)=((1-\alpha)ct+1)^{\frac{1}{1-\alpha}},
$$ 
where $c$ is the constant appearing in the soliton equation \eqref{sol-def-DG} (see \cite{BF}).  The soliton $\gamma$ is therefore called  {\it expanding},  {\it steady}  or  {\it shrinking} depending on whether $c>0$, $c=0$ or $c<0$, and   the maximal time interval of the corresponding self-similar solution is respectively given by,    
$$
\left(-T_\alpha,\infty\right),  \quad (-\infty,\infty),  \quad \left(-\infty,T_\alpha\right),  \qquad\mbox{where}\quad T_\alpha:=\frac{1}{(1-\alpha)|c|}>0,   
$$
often called {\it immortal}, {\it eternal} and {\it ancient} solutions, respectively.  For instance, $\alpha=0$ if $q$ is the Ricci tensor or form of any connection associated to a metric or to an almost-hermitian structure, and $\alpha=\unt$ for most of the flows for $G_2$-structures in the literature.  

We refer to \cite{BF, solitons} and Section \ref{alg-sol-sec} for overviews on solitons for different kinds of geometric flows in complex, symplectic and $G_2$ geometries, including: 

\begin{enumerate}[{\small $\bullet$}]  
\item Ricci flow \cite{libro}.  

\item Chern-Ricci flow \cite{TstWnk}.   

\item Pluriclosed flow \cite{Str}.   

\item Hermitian curvature flow \cite{StrTn}.  

\item Anti-complexified Ricci flow \cite{LeWng}.  

\item Symplectic curvature flow \cite{StrTn2}.    

\item $G_2$-Laplacian flow \cite{Lty, Lin, KrgMckTsu}.   
\end{enumerate}

\begin{remark}
The word {\it soliton} has been used in PDE theory since 19th century in the context of Korteweg-de Vries equation to name certain solutions resembling solitary water waves.  More generally, in the study of geometric flows, solitons refer to geometric structures which evolve along symmetries of the flow (i.e.\ self-similar solutions) and the use of the word soliton was initiated by Hamilton in the 80s (see \cite{Hml}) in the context of Ricci flow to name {\it Ricci solitons}.   Nowadays, solitons are spread over the fields of differential geometry and geometric analysis.  
\end{remark}

\section{Homogeneous geometric structures}\label{hom-sec}

A differentiable manifold endowed with a geometric structure, say $(M,\gamma)$, is said to be {\it homogeneous} if its automorphism group,
$$
\Aut(M,\gamma):=\{ f\in\Diff(M):f^*\gamma=\gamma\},
$$
acts transitively on $M$.  Note that this property makes the points of $M$ geometrically indistinguishable.     

Recall that a {\it homogeneous space} is a quotient $G/K$ of a Lie group $G$ over a closed subgroup $K\subset G$.  If $(M,\gamma)$ is homogeneous, then each Lie group $G\subset\Aut(M,\gamma)$ which is still transitive on $M$ gives rise to a presentation of $M$ as a homogeneous space $G/K$, where $K$ is the isotropy subgroup of $G$ at some {\it origin} point $o\in M$.  In this way, $\gamma$ becomes a $G$-invariant geometric structure on the homogeneous space $M=G/K$.  On the other hand, one can also start with a homogeneous space $M=G/K$ endowed with a $G$-{\it invariant} geometric structure $\gamma$ (i.e.\ $\tau_a^*\gamma=\gamma$ for any $a\in G$, where $\tau_a\in\Diff(M)$ is defined by $\tau_a(bK)=abK$ for all $b\in G$), giving rise to a homogeneous $(M,\gamma)$.  

Any homogeneous space $G/K$ will be assumed in this paper to be only {\it almost-effective}  (i.e.\ the subgroup $\{a\in K:\tau_a=id\}$ is discrete) rather than {\it effective} (i.e.\ $\tau_a=id$ if and only if $a=e$).  Note that $G/K$ is effective if $G\subset\Aut(M,\gamma)$.   

If $\ggo=\kg\oplus\pg$ is a {\it reductive decomposition} (i.e.\ $\Ad(K)\pg\subset\pg$) for the homogeneous space $G/K$, where $\ggo$ and $\kg$ respectively denote the Lie algebras of $G$ and $K$, then the tangent space at the origin of $M$ is identified with $\pg$, 
$$
T_oM\equiv\pg, 
$$ 
via the isomorphism 
\begin{equation}\label{ideT}
\pg\longrightarrow T_oM, \qquad X\mapsto d\pi|_eX=X_o,
\end{equation}
since $\Ker d\pi|_e=\kg$.  Here $\pi:G\rightarrow G/K$ is the usual projection and each $X\in \ggo$ is also viewed as the vector field on $M$ defined by $X_p:=\left. \ddt\right |_{t=0} \tau_{\exp{tX}}(p)$.   

\begin{remark}\label{metg}
In the case when a Riemannian metric $g_\gamma$ is involved in the geometric structure $\gamma$, i.e.\ contained in or uniquely determined by $\gamma$, one obtains the following:  

\begin{enumerate}[{\small $\bullet$}]
\item $\Aut(M,\gamma)$ is a Lie group, as it is a closed subgroup of the Lie group $\Iso(M,g_\gamma)$ of all isometries of the Riemannian manifold $(M,g_\gamma)$.  In particular, $G$ is closed in $\Aut(M,\gamma)$ if and only if $K$ is compact just as in the Riemannian case.    

\item The existence of a reductive decomposition for $G/K$ is guaranteed since $\overline{\Ad(K)}\subset\Gl(\ggo)$ turns out to be compact.  Anyway, one can just work with the identification $T_oM\equiv\ggo/\kg$ if a reductive decomposition is preferred not to be chosen.
\end{enumerate}
\end{remark}

\subsection{Tensors}\label{ten-sec}
Since a $G$-invariant tensor field on a homogeneous space $M=G/K$ is determined by its value at the origin $o$, it can always be identified with a tensor $\gamma$ on $\pg$ which is $\Ad(K)$-{\it invariant}, i.e.\  $\Ad(z)|_\pg\cdot\gamma=\gamma$ for any $z\in K$, or equivalently if $K$ is connected, $\theta(\ad{Z}|_\pg)\gamma=0$ for all $Z\in\kg$.  We are considering here the usual left $\Gl(\pg)$-action and corresponding $\glg(\pg)$-representation $\theta$ on the finite-dimensional vector space 
$$
T^{r,s}\pg:=\{\gamma:\underbrace{\pg\times\dots\times\pg}_{r}\longrightarrow \underbrace{\pg\otimes\dots\otimes\pg}_{s}:\gamma\;\mbox{is multi-linear}\}
$$ 
of tensors of type $(r,s)$, given by 
\begin{align}
h\cdot\gamma:=& h\cdot_s\gamma(h^{-1}\cdot,\dots,h^{-1}\cdot), \quad\forall h\in\Gl(\pg), \label{hact}\\
\theta(A)\gamma:=& \theta_s(A)\gamma - \gamma(A\cdot,\dots,\cdot)-\dots-\gamma(\cdot,\dots,A\cdot), \quad\forall A\in\glg(\pg), \label{Aact}
\end{align}
where $\cdot_s$ and $\theta_s$ denote the $\Gl(\pg)$-action and $\glg(\pg)$-representation on $\pg\otimes\dots\otimes\pg$ ($s$ times), respectively.  Note that $h\cdot\gamma=(h^{-1})^*\gamma$.  

Conversely, it is easy to see that any $\Ad(K)$-invariant tensor on $\pg$ can be walked around $M$ via the $G$-action to become a $G$-invariant tensor field on $M$.  To sum up, the linear map
$$
(\tca^{r,s}M)^G\longrightarrow (T^{r,s}\pg)^K, \qquad \gamma\longmapsto\gamma_o,
$$
is an isomorphism between the vector space $(\tca^{r,s}M)^G$ of all $G$-invariant tensors fields on $M$ and the finite-dimensional vector space $(T^{r,s}\pg)^K$ of all $\Ad(K)$-invariant tensors on $\pg$.  

An alternative identification of $(\tca^{r,s}M)^G$ is given by the isomorphism 
$$
(\tca^{r,s}M)^G\longrightarrow \tca^{r,s}_0G, \qquad \gamma\longmapsto\pi^*\gamma,
$$
where $\tca^{r,s}_0G\subset\tca^{r,s}G$ is the vector subspace of all $(r,s)$-tensors $\gamma$ on the Lie group $G$ such that the following conditions hold: 
\begin{enumerate}[{\rm (i)}]
\item $\gamma$ is left-invariant. 

\item $\gamma$ is invariant under the action of $K$ on $G$ by conjugation. 

\item $\gamma(\dots,Z,\dots)=0$ for any $Z\in\kg$.   
\end{enumerate}
Note that, algebraically, $\tca^{r,s}_0G$ is identified with
$$
\tca^{r,s}_0\ggo:=\left\{ \gamma\in T^{r,s}\ggo:\gamma\;\mbox{is $\Ad(K)$-invariant and (iii) holds}\right\},
$$
and the linear map 
$$
(T^{r,s}\pg)^K\longrightarrow T^{r,s}_0\ggo, \qquad \gamma\longmapsto d\pi|_e^*\gamma,
$$
is an isomorphism.  In the context of differential forms, the differentials of forms $d_G$ and $d_M$ on $G$ and $M$, respectively, satisfy that
$$
d_G\left(d\pi|_e\right)^*\gamma = \left(d\pi|_e\right)^*d_M\gamma, \qquad \forall \gamma\in (\Lambda^k\pg^*)^K\subset (T^{k,0}\pg)^K,  
$$
and thus
$$
d_M\gamma(X_1,\dots,X_{k+1}) =\sum_{i<j} (-1)^{i+j}\gamma([X_i,X_j]_\pg,X_1,\dots,\widehat{X}_i,\dots,\widehat{X}_j,\dots,X_k),
$$
for any $X_1,\dots,X_{k+1}\in\pg$, where $[X_i,X_j]_\pg$ is the projection of the Lie bracket $[X_i,X_j]$ on $\pg$ relative to $\ggo=\kg\oplus\pg$ and $\widehat{X}_i$ means that $X_i$ has to be deleted.

\subsection{Equivalence}\label{equiv-sec}
A diffeomorphism of a homogeneous space $M=G/K$ is called {\it equivariant} if it is given by an automorphism of $G$ taking $K$ onto $K$.  Let $\Aut(G/K)$ denote the group of all equivariant diffeomorphisms of $M=G/K$.  It is easy to check that $\Aut(G/K)$ acts on $(\tca^{r,s}M)^G$ and two $G$-invariant geometric structures (or tensor fields) are said to be {\it equivariantly equivalent} when they belong to the same $\Aut(G/K)$-orbit.    

In the case when $G$ is simply connected and $K$ connected (in particular, $M$ simply connected), which will be assumed in the rest of this subsection, $\Aut(G/K)$ can be identified with the Lie group
$$
\Aut(\ggo/\kg):=\left\{ \overline{h}\in\Aut(\ggo):\overline{h}(\kg)=\kg\right\},
$$ 
with Lie algebra $\Der(\ggo/\kg):=\{ D\in\Der(\ggo):D(\kg)\subset\kg\}$.  At the Lie algebra level, $\Aut(\ggo/\kg)$ acts on $(T^{r,s}\pg)^K$ by 
$$
\gamma\mapsto h\cdot\gamma=(h^{-1})^*\gamma, \qquad \mbox{where}\quad \overline{h}=\left[\begin{matrix} \ast&\ast\\ 0&h\end{matrix}\right]\in\Aut(\ggo/\kg),
$$
giving rise to the equivariant equivalence in the simply connected case.  Note that $\Ad{K}\subset\Aut(\ggo/\kg)$.  

\begin{remark}
Two non-equivariantly equivalent $G$-invariant geometric structures $\gamma$ and $\gamma'$ on a homogeneous space $M=G/K$ may be still equivalent in the general sense as in Section \ref{intro}, that is, $\gamma'=f^*\gamma$ for some $f\in\Diff(M)$.  
\end{remark} 

The computation of $(T^{r,s}\pg)^K$ is a technical problem in representation theory whose difficulty strongly depends on the decomposition of $\pg$ into irreducible and isotypic components and the type of each irreducible component $\pg_i$.  Recall that $\pg_i$ is said to be of {\it real}, {\it complex} or {\it quaternionic type}, depending on whether the space of intertwining operators $\End_K(\pg_i)$ is $\RR$, $\CC$ or $\HH$, respectively.  However, the existence problem of a structure $\gamma\in(T^{r,s}\pg)^K$ satisfying in addition certain non-degeneracy condition can be very tricky.  On the other hand, the moduli space 
$$
\Gamma/\Aut(\ggo/\kg)\subset(T^{r,s}\pg)^K/\Aut(\ggo/\kg)
$$ 
of a space $\Gamma$ of $G$-invariant geometric structures on $M=G/K$, up to equivariant equivalence, is an object much harder to compute or understand.  

A larger group acting on $(T^{r,s}\pg)^K$ is the normalizer
$$
N_{\Gl(\pg)}(\ad{\kg}|_\pg):=\{ h\in\Gl(\pg):h\ad{\kg}|_\pg h^{-1}\subset\ad{\kg}|_\pg\}.  
$$
This follows from the fact that  
$$
\theta(\ad{\kg}|_\pg)h\cdot\gamma =  h\cdot\theta(h^{-1}\ad{\kg}|_\pg h)\gamma =0, \qquad\forall\; h\in N_{\Gl(\pg)}(\ad{\kg}|_\pg), \quad\gamma\in (T^{r,s}\pg)^K.
$$  
Note that $\Aut(\ggo/\kg)|_\pg$ and the centralizer $C_{\Gl(\pg)}(\ad{\kg}|_\pg)$ are both subgroups of $N_{\Gl(\pg)}(\ad{\kg}|_\pg)$.  This action is many times a useful tool to obtain non-equivariantly equivalent $G$-invariant geometric structures from a given one.

\subsection{Operators versus tensors}\label{ope-sec}
A property that is satisfied by many classes of geometric structures is that the orbit $\Gl(\pg)\cdot\gamma$
is open in a certain vector subspace $T\subset T^{r,s}\pg$ and consists precisely of those tensors which are non-degenerate in some sense.  In that case, $\theta(\glg(\pg))\gamma = T$, and thus if we consider the $\Gl(\pg)_\gamma$-invariant decomposition,  
\begin{equation}\label{dec-g}
\glg(\pg)=\glg(\pg)_\gamma\oplus\qg_\gamma,
\end{equation}
then $\theta(\qg_\gamma)\gamma=T$, where $\glg(\pg)_\gamma:=\{ A\in\glg(\pg):\theta(A)\gamma=0\}$ is the Lie algebra of the stabilizer subgroup $\Gl(\pg)_\gamma$.  Moreover, for every tensor $q\in T$, there exists a unique operator $Q\in\qg_\gamma$ such that
\begin{equation}\label{nondeg2}
q=\theta(Q)\gamma.
\end{equation}
Note that $\theta(I)\gamma=(s-r)\gamma$ for any $\gamma\in T^{r,s}\pg$ (see \eqref{Aact}).  

Some examples of geometric structures for which this translation to operators works follow (see \cite{BF, solitons} for further information).  

\begin{enumerate}[{\small $\bullet$}]
\item A Riemannian metric $\gamma=g$:  
$$
T=\sca^2\pg^*\subset T^{2,0}\pg, \qquad \glg(\pg)_\gamma=\sog(\pg), \qquad \qg_\gamma=\sym(\pg):=\{ A\in\glg(\pg):A^t=A\}.
$$
For example, if $q=-2\ricci(g)$, where $\ricci(g)$ is the Ricci tensor of $g$, then $Q=\Ricci(g)$, the Ricci operator of $g$.

\item A pseudo-Riemannian metric $\gamma=g$ of signature $(p,q)$ ($\dim{\pg}=p+q$): $T=\sca^2\pg^*\subset T^{2,0}\pg$,
\begin{align*}
\glg(\pg)_\gamma=&\{ A\in\glg(\pg):g(A\cdot,\cdot)=-g(\cdot,A\cdot)\}\simeq\sog(p,q), \\ \qg_\gamma=&\{ A\in\glg(\pg):g(A\cdot,\cdot)=g(\cdot,A\cdot)\}.
\end{align*}

\item An almost-symplectic structure $\gamma=\omega$ ($\dim{\pg}=2n$): $T=\Lambda^2\pg^*\subset T^{2,0}\pg$,
\begin{align*}
\glg(\pg)_\gamma=&\{ A\in\glg(\pg):\omega(A\cdot,\cdot)=-\omega(\cdot,A\cdot)\}\simeq\spg(n,\RR), \\ \qg_\gamma=&\{ A\in\glg(\pg):\omega(A\cdot,\cdot)=\omega(\cdot,A\cdot)\}.
\end{align*}

\item An almost-hermitian structure $\gamma=(\omega,g,J)$, $\omega=g(J\cdot,\cdot)$ ($\dim{\pg}=2n$):
\begin{align*}
T=&\left\{ (\overline{\omega},\overline{g})\in\Lambda^2\pg^*\times\sca^2\pg^* : \overline{g}^{1,1}=\overline{\omega}^{1,1}(\cdot,J\cdot)\right\}\subset T^{2,0}\pg\times T^{2,0}\pg, \\
\glg(\pg)_\gamma=&\{ A\in\sog(\pg):AJ=JA\}\simeq\ug(n), \qquad \qg_\gamma=\qg_1\oplus\qg_2\oplus\qg_3,
\end{align*}
where,
\begin{align*}
\qg_1=&\{ A\in\glg(\pg):A^t=A, \quad AJ=-JA\}, \\ 
\qg_2=&\{ A\in\glg(\pg):A^t=-A, \quad AJ=-JA\}, \\ 
\qg_3=&\{ A\in\glg(\pg):A^t=A, \quad AJ=JA\}.
\end{align*}
Note that $\spg(n,\RR)=\ug(n)\oplus\qg_1$, $\sog(2n)=\ug(n)\oplus\qg_2$ and $\glg_n(\CC)=\ug(n)\oplus\qg_3$. 

\item A $G_2$-structure $\gamma=\vp$ ($\dim{\pg}=7$):  
$$
T=\Lambda^3\pg^*\subset T^{3,0}\pg, \qquad \glg(\pg)_\gamma=\ggo_2, \qquad \qg_\gamma=\qg_1\oplus\qg_7\oplus\qg_{27},
$$
where $\qg_1=\RR I$, $\sog(\pg)=\ggo_2\oplus\qg_7$ and $\qg_{27}=\sym_0(\pg):=\{ A\in\sym(\pg):\tr{A}=0\}$.  Here, an example of preferred direction is $q(\vp)=\Delta_\vp\vp$, where $\Delta_\vp$ is the Hodge-Laplacian on $3$-forms (giving rise to the $G_2$-Laplacian flow), and $Q_\vp$ can be written in terms of the Ricci operator and the torsion (see \cite[Section 3]{LS-ERP}).     
\end{enumerate}

\section{Solitons on homogeneous spaces}\label{alg-sol-sec}

In this section, we study solitons on a fixed homogeneous space $M=G/K$.  As in Section \ref{intro}, we consider a space of geometric structures $\Gamma$ and assume that every $\gamma\in\Gamma$ is $G$-invariant, that is, 
$$
\Gamma\subset (T^{r,s}\pg)^K,
$$
where $\ggo=\kg\oplus\pg$ is any reductive decomposition as in Section \ref{ten-sec}.  Accordingly, the equivalence between $G$-invariant structures on $M$ is defined by scalings and equivariant diffeomorphisms of $M$ (see Section \ref{equiv-sec}).  Let us assume that $\Gamma$ is $\Aut(G/K)$-invariant.     

As a preferred direction $q$ we can choose any direction from the general case (see Section \ref{intro}), which produces a $G$-invariant preferred direction determined by a function 
$$
\gamma\mapsto q(\gamma)\in T_\gamma\Gamma\subset (T^{r,s}\pg)^K.  
$$ 
In this light, once the space $\Gamma$ and the preferred direction $q$ have been chosen, we obtain that if $\gamma\in\Gamma$ is a soliton in the sense of \eqref{sol-def}, called a {\it semi-algebraic soliton} in the literature, then   
\begin{equation}\label{semi-alg-sol}
q(\gamma)=c\gamma+\theta(D_\pg)\gamma, \qquad \mbox{for some} \quad c\in\RR, \quad 
D=\left[\begin{matrix} \ast&\ast\\ 0&D_\pg\end{matrix}\right]\in\Der(\ggo/\kg), 
\end{equation}
where $\theta$ denotes the usual $\glg(\pg)$-representation on tensors (see \eqref{Aact}).  Indeed, we have that  
\begin{equation}\label{lieD}
\theta(D_\pg)\gamma=-\lca_{X_D}\gamma = -\left. \ddt\right |_{t=0} f(t)^*\gamma \in T_\gamma\Aut(G/K)\cdot\gamma, 
\end{equation}
where $f(t)\in\Aut(G/K)$, $D=\left. \ddt\right |_{t=0} df(t)|_e\in\Der(\ggo/\kg)$ and $\lca_{X_D}$ denotes the Lie derivative with respect to the vector field $X_D\in\mathfrak{X}(M)$ defined by 
$$
X_D(p)=\left. \ddt\right |_{t=0} f(t)(p), \qquad\forall p\in M.
$$ 
In particular, $(G/K,\gamma)$ is of course a soliton from the general point of view considered in \eqref{sol-def-DG}.    

Conversely, in the simply connected case, any $(G/K,\gamma)$ satisfying \eqref{semi-alg-sol} is in fact a semi-algebraic soliton.  Indeed, the existence of the automorphisms $f(t)\in\Aut(G/K)$ such that $df(t)|_e=e^{tD}\in\Aut(\ggo/\kg)$ is guaranteed for any $D\in\Der(\ggo/\kg)$.   Note that the fact that $\theta(D_\pg)\gamma\in T_\gamma\Gamma$ automatically follows from \eqref{semi-alg-sol}.  

\begin{remark}
The vector fields $X_D$'s described above can be viewed as  generalizations of {\it linear vector fields} on $\RR^n$ (i.e.\ $X_v=Av$, $A\in\glg_n(\RR)$) and have been strongly used in control theory on Lie groups after the pioneering article \cite{AylTir}.   Since $X_D(o)=0$, none of these vector fields can belong to $\pg$ (see \eqref{ideT}).   
\end{remark}

Unlike the concept of soliton, which is invariant under pull-back by diffeomorphisms, the concept of semi-algebraic soliton is not a geometric invariant, as it may depend on the presentation of the homogeneous geometric structure $(M,\gamma)$ as a homogeneous space $(G/K,\gamma)$.  

On a given homogeneous space $M=G/K$, the corresponding geometric flow  
\begin{equation}\label{flow-LG}
\ddt\gamma(t)=q(\gamma(t)), \qquad \gamma(0)=\gamma,
\end{equation}
is actually an ODE rather than a PDE.  Indeed, the solutions to \eqref{flow-LG} are precisely the integral curves of the vector field $q$ on the finite-dimensional vector space $(T^{r,s}\pg)^K$.  In particular, short time existence (forward and backward) of $G$-invariant solutions and their uniqueness (among $G$-invariant solutions) are always guaranteed.  

We assume from now on that $G$ is simply connected and $K$ is connected.  Notice that $\gamma$ is a semi-algebraic soliton if and only if the solution $\gamma(t)$ to \eqref{flow-LG} is given by
\begin{equation}\label{sas-ss}
\gamma(t)=c(t)f(t)^*\gamma, \qquad\mbox{for some}\quad c(t)\in\RR^*, \quad f(t)\in\Aut(G/K),
\end{equation}
that is, $\gamma(t)$ is self-similar relative to equivariant equivalence (cf.\ \eqref{sss}).  More precisely, it is easy to check that condition \eqref{semi-alg-sol} holds for $(G/K,\gamma)$ if and only if 
\begin{equation}\label{sas-evg}
\gamma(t) = c(t) e^{s(t)D_\pg}\cdot\gamma, 
\end{equation}
is a solution to \eqref{flow-LG}, where 
$$
c(t):=\left((1-\alpha)ct+1\right)^{\frac{1}{1-\alpha}}, \qquad s(t):=\frac{\log((1-\alpha)ct+1)}{(1-\alpha)c} \quad  (\mbox{set}\; s(t):=t \;\mbox{if}\; c=0).   
$$
Recall that $\alpha<1$ (see \eqref{alfa}) and note that $c(0)=1$, $c'(t)=c\, c(t)^\alpha$ and $s(0)=0$, $s'(t)=c(t)^{\alpha-1}$.   

The following result is essentially contained in \cite[Theorem 3.1]{Jbl}.  

\begin{proposition}\label{Kcomp}
Let $(G/K,\gamma)$ be a semi-algebraic soliton and assume that $K$ is compact.  Then there exist $f(t)\in\Aut(G/K)$, $f(0)=id$, such that,
\begin{enumerate}[{\rm (i)}]
\item the solution to \eqref{flow-LG} is given by $\gamma(t)=c(t)f(t)^*\gamma$, for some $c(t)\in\RR^*$;
\item $f(t)|_K=id$ for all $t$.
\end{enumerate}
\end{proposition}

\begin{proof}
According to \eqref{sas-ss}, $\gamma(t)=c(t)g(t)^*\gamma$, for some $c(t)\in\RR^*$ and $g(t)\in\Aut(G/K)$.  We can assume that $g(0)=id$ by just composing with $g(0)^{-1}$, which belongs to $\Aut(M,\gamma)$ as $\gamma(0)=\gamma$.  Since $K$ is connected and compact, the identity component $\Aut(K)_0$ of $\Aut(K)$ consists of inner automorphisms, and thus there exist $a(t)\in K$, $a(0)=e$, such that $g(t)|_K=I_{a(t)}$, the conjugation map by $a(t)$.  But $I_{a(t)}\in\Aut(G/K)$ and is an automorphism of $\gamma$ (recall that $dI_{a(t)}|_o=\Ad(a(t))\in\Ad(K)$), so $f(t):=I_{a(t)}^{-1}\circ g(t)$ satisfies that $f(t)^*\gamma=g(t)^*\gamma$ and $f(t)|_K=id$ for all $t$, concluding the proof.  
\end{proof}

In the case when the reductive decomposition $\ggo=\kg\oplus\pg$ for $G/K$ with $B(\kg,\pg)=0$ is considered, where $B$ is the Killing form of $\ggo$, it follows that $D\pg\subset\pg$ for any $D\in\Der(\ggo/\kg)$ (see \cite[Lemma 3.10]{homRS}).  If in addition $K$ is compact, then by Proposition \ref{Kcomp} the derivation in the definition \eqref{semi-alg-sol} of a semi-algebraic soliton $(G/K,\gamma)$ takes the simpler form 
\begin{equation}\label{D00}
D=\left[\begin{matrix} 0&0\\ 0&D_\pg\end{matrix}\right]\in\Der(\ggo/\kg).
\end{equation}

Recall that $\Gamma$ is typically contained in a single $\Gl(\pg)$-orbit.  We assume for the rest of this section that $\Gl(\pg)\cdot\gamma$ is in addition open in a suitable subspace $T\subset T^{r,s}\pg$, as in Section \ref{ope-sec}.  In that case, in terms of the operator viewpoint proposed in Section \ref{ope-sec}, the semi-algebraic soliton condition \eqref{semi-alg-sol} for $(G/K,\gamma)$ can be rewritten as 
\begin{equation}\label{sas-Q} 
Q_\gamma = \tfrac{c}{s-r} I + \proy_\gamma(D_\pg), 
\end{equation}  
where $q(\gamma)=\theta(Q_\gamma)\gamma$, $Q_\gamma\in\qg_\gamma$ (see \eqref{nondeg2}) and $\proy_\gamma:\glg(\pg)\rightarrow\qg_\gamma$ is the projection relative to decomposition \eqref{dec-g}.  The matrix equation \eqref{sas-Q} has been extremely useful to explore the existence and study the structure of semi-algebraic solitons in many different contexts.    

In the special case when \eqref{sas-Q} holds for a derivation $D\in\Der(\ggo/\kg)$ such that 
\begin{equation}\label{as-Q} 
\proy_\gamma(D_\pg)=D_\pg, 
\end{equation} 
$(G/K,\gamma)$ is called an {\it algebraic soliton}.  It is worth noticing that being an algebraic soliton may a priori not only depend on the homogeneous space presentation, but also on the reductive decomposition $\ggo = \kg \oplus \pg$ chosen.  Algebraic solitons are also distinguished from other points of view, as bracket flow evolution (see Section \ref{mba-sec}) and the flow diagonal property (see Section \ref{fd-sec}).   
 
Any homogeneous Ricci soliton is isometric to an algebraic soliton (see \cite{Jbl2}).  On the other hand, examples of $G_2$-Laplacian and pluriclosed semi-algebraic solitons which are not isometric to any algebraic soliton were respectively found in \cite{Ncl,LF} and \cite{ArrLfn}. 

Interplaying geometric and algebraic aspects of the homogeneous structure $(G/K,\gamma)$, the role of algebraic solitons has been crucial in the study of homogeneous Ricci solitons.  On the other hand, they have also used to obtain explicit examples of homogeneous soliton geometric structures in the case of pseudo-Riemannian, complex, symplectic and $G_2$ geometries.  A far from complete list of references follows.    

\begin{enumerate}[{\small $\bullet$}]  
\item Ricci solitons \cite{Nkl, Frn, solvsolitons, Wll, KdgPyn, Jbl, homRS, Jbl2, alek, JblPtrWll, ArrLfn2} (see the survey \cite{cruzchica} for references previous to 2008).  

\item Pseudo-Riemannian Ricci solitons \cite{ClvGrc, ClvFin, Ond}.  

\item Chern-Ricci solitons \cite{SCF, CRF}.   

\item Pluriclosed solitons \cite{ArrLfn}.   

\item Hermitian curvature flow solitons \cite{Puj, LfnPujVzz}.  

\item Anti-complexified Ricci solitons \cite{minimal, Frn2}.  

\item Symplectic curvature flow solitons \cite{SCF, Frn2, SCFmuA}.    

\item $G_2$-Laplacian solitons \cite{Ncl, LF, LS-ERP, FinRff3, FinRff, BggFin, ERP, MrnSrp}.   \end{enumerate}

\begin{remark}
Given a semi-algebraic soliton $(G/K,\gamma)$, if $G$ admits a cocompact discrete subgroup $\Lambda$, then $\Lambda$ acts freely and properly discontinuous on $G/K$ and the solution $\gamma(t)$ is also a solution to \eqref{flow-LG} on the compact manifold $M'=\Lambda \backslash G/K$, which is locally diffeomorphic to $M=G/K$.  However, the locally homogeneous manifold $(M',\gamma)$ is not longer a soliton in general since the field $X_D$ may not descend to $M'$.  The solution $(M,\gamma(t))$ is very peculiar though, it is `locally self-similar' in the sense that $\gamma(t)$ is locally equivalent to $\gamma$ up to scaling for all $t$.  
\end{remark}

\section{On the evolution of homogeneous geometric structures}\label{evol-sec} 

We assume in this section that $\Aut(M,\gamma)$ is a Lie group.  Recall that this holds, for example, when there is a Riemannian metric attached to $\gamma$ (see Remark \ref{metg}).

\begin{proposition}\label{exist}
For any homogeneous geometric structure $(M,\gamma)$, there exists a unique solution $\gamma(t)$ to \eqref{flow-DG} which is $\Aut(M,\gamma)$-invariant.  Furthermore, $\gamma(t)$ has the following properties: 

\begin{enumerate}[{\rm (i)}]
\item $\gamma(t)$ is defined in a maximal interval of time $(T_-,T_+)$, where $T_-<0<T_+$.

\item $\Aut(M,\gamma(t))=\Aut(M,\gamma)$ for all $t\in(T_-,T_+)$.

\item For any transitive Lie subgroup of automorphisms $G\subset\Aut(M,\gamma)$, $\gamma(t)$ is the unique $G$-invariant solution to \eqref{flow-LG} on the corresponding homogeneous space $M=G/K$.
\end{enumerate}
\end{proposition}

\begin{proof}
The existence and uniqueness of the solution $\gamma(t)$ and part (i) follow from the paragraph below \eqref{flow-LG} applied to $G=\Aut(M,\gamma)$.  On the other hand, part (iii) follows from uniqueness and the fact that $\gamma(t)$ is also $G$-invariant for any $G\subset\Aut(M,\gamma)$.  This implies that given $t_0$, since $\Aut(M,\gamma)\subset\Aut(M,\gamma(t))$ for all $t$, we have that $\gamma(t+t_0)$ is the $\Aut(M,\gamma)$-invariant solution starting at $\gamma(t_0)$.  Thus $\gamma(t+t_0)$ is also $\Aut(M,\gamma(t_0))$-invariant and so $\Aut(M,\gamma(t_0))\subset\Aut(M,\gamma)$, by evaluating at $t=-t_0$, and part (ii) holds, concluding the proof.
\end{proof}

Since the uniqueness of geometric flow solutions is in most cases established only for compact manifolds, we do not know a priori if there are homogeneous solutions other than $\gamma(t)$ starting at a noncompact homogeneous $(M,\gamma)$.  This seems to be very unlikely though.

The following result was proved in \cite[Theorem 3.1]{Jbl} for Ricci solitons; we essentially follow the lines of that proof.

\begin{theorem}\label{sas-hom}
Let $(M,\gamma)$ be a soliton as in \eqref{sol-def-DG} which is homogeneous.  If the solution $\gamma(t)$ (see Proposition \ref{exist}) is self-similar as in \eqref{sss}, then $(G/K,\gamma)$ is a semi-algebraic soliton for the presentation $M=G/K$ with $G=\Aut(M,\gamma)$.  
\end{theorem}

\begin{remark}
In particular, $\gamma(t)=c(t)f(t)^*\gamma$ for some $c(t)\in\RR^*$ and $f(t)\in\Aut(G/K)$ such that $f(0)=id$ and $f(t)|_K=id$ for all $t$ by Proposition \ref{Kcomp}, since the isotropy $K$ is compact.  
\end{remark}

\begin{proof}
Consider the presentation $M=G/K$, where $G=\Aut(M,\gamma)$ and $K$ is the isotropy subgroup of $G$ at some $p\in M$.  It follows from \eqref{sss} that $\gamma(t)=c(t)g(t)^*\gamma$ for some $c(t)\in\RR^*$ and $g(t)\in\Diff(M)$.  It can be assumed that $g(0)=id$, and also that $g(t)(p)=p$ for all $t$, by composing with $h(t)\in G$ such that $h(t)^{-1}(p)=g(t)(p)$.  Since $G=\Aut(M,\gamma(t))=g(t)^{-1}Gg(t)$ for all $t$ (see Proposition \ref{exist}, (ii)), we can define an isomorphism $\tilde{f}(t):G\longrightarrow G$ by $\tilde{f}(t)(a):=g(t)ag(t)^{-1}$ for all $a\in G$, which in turns determines $f(t)\in\Aut(G/K)$ as $\tilde{f}(t)(K)\subset K$.  But $f(t)=g(t)$ on $G/K$ for all $t$:
$$
f(t)(aK) = g(t)ag(t)^{-1}K = g(t)ag(t)^{-1}(p) = g(t)(a(p)) = g(t)(aK), \qquad\forall a\in G, 
$$
which concludes the proof.  
\end{proof}

\begin{remark}\label{fmore}
It follows from the proof of the above proposition that if $\gamma(t)=c(t)f(t)^*\gamma$ for some $c(t)\in\RR^*$ and $f(t)\in\Diff(M)$ such that $f(0)=id$, $f(t)(o)=o$, then $f(t)\in\Aut(G/K)$ for all $t$ for $G=\Aut(M,\gamma)$.  
\end{remark}

The uniqueness of solutions has been proved in some cases (e.g.\ for the Ricci flow) among bounded curvature solutions on non-compact manifolds.  We note that this applies to homogeneous solutions.  

\begin{corollary}\label{sas-hom-cor}
Assume that for each homogeneous $(M,\gamma)$, there exists a unique solution $\gamma(t)$ to \eqref{flow-DG} such that $(M,\gamma(t))$ is homogeneous for all $t$.  Then any homogeneous soliton $(M,\gamma)$ is a semi-algebraic soliton when presented as a homogeneous space $(G/K,\gamma)$ with $G=\Aut(M,\gamma)$.
\end{corollary}

\section{The moving-bracket approach}\label{mba-sec}

Let $(G/K,\gamma)$ be a homogeneous space endowed with a $G$-invariant geometric structure.  We note that   all the geometry of $(G/K,\gamma)$ is essentially encoded only in the tensor $\gamma\in (T^{r,s}\pg)^K$ and the Lie bracket $\mu$ of $\ggo$.  Therefore, in order to vary homogeneous geometric structures, it is natural to wonder about what if we vary $\mu$ rather than $\gamma$.  To do so, we consider a fixed $(k+n)$-dimensional real vector space $\ggo$ together with a direct sum decomposition
\begin{equation}\label{fixdec}
\ggo=\kg\oplus\pg, \qquad \dim{\kg}=k, \qquad \dim{\pg}=n,
\end{equation}
a fixed suitable tensor $\gamma\in T^{r,s}\pg$ (or a tuple of tensors) and the {\it variety of Lie algebras} $\lca\subset\Lambda^2\ggo^*\otimes\ggo\subset T^{2,1}\ggo$ (i.e.\ the algebraic subset of all Lie brackets on the vector space $\ggo$).  In this way, each $\mu\in\lca$ determines a simply connected almost-effective homogeneous space endowed with an invariant geometric structure with reductive decomposition $\ggo=\kg\oplus\pg$, say $(G_\mu/K_\mu,\gamma)$,  provided that the following conditions hold:

\begin{enumerate}[{\rm (i)}]
\item  $\mu(\kg,\kg)\subset\kg$ and $\mu(\kg,\pg)\subset\pg$.

\item $K_\mu$ is closed in $G_\mu$, where $G_\mu$ denotes the simply connected Lie group with Lie algebra $(\ggo,\mu)$ and $K_\mu$ is the connected Lie subgroup of $G_\mu$ with Lie algebra $\kg$.

\item $\{ Z\in\kg:\mu(Z,\pg)=0\}=0$. 

\item $\theta(\ad_{\mu}{Z}|_{\pg})\gamma=0$ for all $Z\in\kg$.
\end{enumerate}
If $\hca_{k,n}^\gamma$ denotes the subset of $\lca$, consisting of the elements satisfying (i)-(iv), then via the identification  
\begin{equation}\label{hsmu}
\hca_{k,n}^\gamma\ni\mu\longleftrightarrow (G_{\mu}/K_{\mu},\gamma),
\end{equation}
the space $\hca_{k,n}^\gamma$ parametrizes the set of all $n$-dimensional simply connected homogeneous spaces with $k$-dimensional isotropy endowed with an invariant geometric structure of the same type as the fixed $\gamma$ (see \cite{spacehm,BF} for more detailed treatments).    

Recall that any Lie group isomorphism $G\rightarrow G'$ taking $K$ onto $K'$ defines a diffeomorphism $G/K\rightarrow G'/K'$ called an {\it equivariant diffeomorphism}.  In that case, the homogeneous spaces are said to be {\it equivariantly equivalent}.   

The natural actions on tensors provide the following key equivariant equivalence between geometric structures: 
\begin{equation}\label{equiv}
(G_{\overline{h}\cdot\mu}/K_{\overline{h}\cdot\mu},\gamma)  
\longleftarrow (G_\mu/K_\mu,h^*\gamma), \qquad\forall\; \overline{h}=\left[\begin{matrix} h_\kg&0\\ 0&h \end{matrix}\right]\in\Gl(\ggo),
\end{equation}
given by the Lie group isomorphism $G_{\mu}\rightarrow G_{\overline{h}\cdot\mu}$ with derivative $\overline{h}$.  In the case when $\Gamma\subset\Gl(\pg)\cdot\gamma$, it follows from \eqref{hsmu} and \eqref{equiv} that the orbit $(\Gl(\kg)\times\Gl(\pg))\cdot\mu$ contains all the geometric structures of the same type of $\gamma$ (up to equivalence) on the homogeneous space $G_\mu/K_\mu$, for each $\mu\in\hca_{k,n}^\gamma$.  Note that all the homogeneous spaces involved share the same reductive decomposition $\ggo=\kg\oplus\pg$.    

The scaling of tensors can be viewed on the space $\hca_{k,n}^\gamma$ as follows.  Since $(cI)^*\gamma=c^{r-s}\gamma$ for any $c\in\RR^*$  and $\gamma\in T^{r,s}\pg$ (see \eqref{hact}), if we set $\overline{h}:=\left(I,c^{-1}I\right)\in\Gl(\kg)\times\Gl(\pg)$, then we obtain from \eqref{equiv} that 
$$
(G_{c\cdot\mu}/K_{c\cdot\mu},\gamma)\simeq (G_{\mu}/K_{\mu},c^{s-r}\gamma), 
$$
where $c\cdot\mu:=\overline{h}\cdot\mu$ is given by
\begin{equation}\label{scmu}
c\cdot\mu|_{\kg\times\kg}=\mu, \qquad c\cdot\mu|_{\kg\times\pg}=\mu, \qquad c\cdot\mu|_{\pg\times\pg}=c^2\mu_{\kg}+c\mu_{\pg},  
\end{equation}
and the subscripts denote the $\kg$- and $\pg$-components of $\mu|_{\pg\times\pg}$.  

\begin{remark}\label{conv} 
The usual convergence of a sequence of brackets produces convergence of the corresponding geometric structures in well-known senses, like pointed (or Cheeger-Gromov) and smooth up to pull-back by diffeomorphisms, under suitable conditions (see \cite{spacehm,BF}).  Remarkably, a degeneration (i.e.\ $\lambda\in\overline{\Gl(\ggo)\cdot\mu}\setminus\Gl(\ggo)\cdot\mu$) may give rise to the convergence of a sequence of geometric structures on a given homogeneous space (see \eqref{equiv}) toward a structure on a different homogeneous space, which may be topologically quite different.   
\end{remark}

The moving-bracket approach has actually been used for decades in homogeneous geometry (see the overviews given in \cite[Section 5]{BF} and \cite[Section 3]{RNder}).  In most applications, concepts and results from geometric invariant theory, including moment maps and their convexity properties, closed orbits, stability, categorical quotients and Kirwan stratification, have been exploited in one way or another.

\subsection{The bracket flow} 
We assume for the rest of the paper that $\Gl(\pg)\cdot\gamma$ is open in $T\subset T^{r,s}\pg$ as in Section \ref{ope-sec}.  

Motivated by equivalence \eqref{equiv}, a main tool to study the geometric flow \eqref{flow-LG} is a dynamical system defined on the variety of Lie algebras $\lca$ called the {\it bracket flow}, defined by
\begin{equation}\label{BF}
\ddt\mu(t) = \theta\left(\left[\begin{matrix} 0&0\\ 0&Q_{\mu(t)} \end{matrix}\right]\right)\mu(t), \qquad\mu(0)=\mu,
\end{equation}
where $Q_\mu\in\qg_\gamma\subset\glg(\pg)$ is the unique operator such that $\theta(Q_\mu)\gamma = q(G_\mu/K_\mu,\gamma)$ (see \eqref{nondeg2}).  

The bracket flow is equivalent to the geometric flow \eqref{flow-LG} in the following precise sense.  Given $\mu\in\hca_{k,n}^\gamma$, we consider the one-parameter families
$$
(G_\mu/K_\mu,\gamma(t)) \qquad \mbox{and}\qquad \left(G_{\mu(t)}/K_{\mu(t)},\gamma\right),
$$
where $\gamma(t)$ is the solution to the geometric flow \eqref{flow-LG} and $\mu(t)$ is the solution to the bracket flow \eqref{BF}.  Recall that $\ggo=\kg\oplus\pg$ is a reductive decomposition for each of the homogeneous spaces involved.

\begin{theorem}\label{BF-thm}\cite[Theorem 5]{BF}
There exist equivariant diffeomorphisms 
$$
\tilde{h}(t):G_\mu/K_\mu\longrightarrow G_{\mu(t)}/K_{\mu(t)} \qquad\mbox{such that}\qquad
\gamma(t)=\tilde{h}(t)^*\gamma, \qquad\forall t.
$$
Moreover, each $\tilde{h}(t)$ can be chosen to be the Lie group isomorphism $G_\mu\longrightarrow G_{\mu(t)}$ with derivative 
$$
\overline{h}(t):=\left[\begin{matrix} I&0\\ 0&h(t) \end{matrix}\right]:\ggo\longrightarrow\ggo, 
$$ 
where $h(t)=d\tilde{h}(t)|_o:\pg\longrightarrow\pg$ is the solution to any of the following ODE's:
\begin{enumerate}[{\rm (i)}]
\item $\ddt h(t)=-h(t)Q_{\gamma(t)}$, $h(0)=I$, where $Q_{\gamma(t)}\in\qg_{\gamma(t)}\subset\glg(\pg)$ is defined by
$$
\theta(Q_{\gamma(t)})\gamma(t)=q(G/K,\gamma(t)).
$$
\item $\ddt h(t)=-Q_{\mu(t)} h(t)$, $h(0)=I$, where $Q_{\mu(t)}\in\qg_{\gamma}\subset\glg(\pg)$ is defined by
$$
\theta(Q_{\mu(t)})\gamma=q(G_{\mu(t)}/K_{\mu(t)},\gamma).
$$
\end{enumerate}
\end{theorem}

\begin{remark}
The following conditions also hold: 
\begin{enumerate}
\item[(iii)] $\gamma(t)=h(t)^*\gamma$ for all $t$. 

\item[(iv)] $\mu(t)=\overline{h}(t)\cdot\mu$ for all $t$. 
\end{enumerate}
\end{remark}

The bracket flow is useful to better visualize the possible limits of solutions, under diverse rescalings, as well as to address regularity issues.  Since the flows differ only by pullback by time-dependent diffeomorphisms, the behavior of any geometric quantity can be studied along the bracket flow solutions.  We now list some of its properties and applications.  

\begin{enumerate}[{\small $\bullet$}] 
\item The flows \eqref{flow-LG} and \eqref{BF} are also equivalent in the following strong sense: each one can be obtained from the other by solving the corresponding ODE in part (i) or (ii) and applying parts (iv) or (iii), accordingly.  In particular, the maximal interval of time $(T_-,T_+)$ where a solution exists is exactly the same for both flows.

\item The velocity of the flow $q(\gamma(t))$ must blow up at any finite-time singularity.  More precisely, if $r>s$ and $T_+<\infty$, then
$$
|q(\gamma(t))|_t\geq\frac{C}{T_+-t}, \qquad |\mu(t)|\geq\frac{C}{(T_+-t)^{\frac{1}{(r-s)(1-\alpha)}}},\qquad \forall t\in[0,T_+),
$$
for some $C>0$ depending only on the dimension $k+n$ and the geometric flow (see \cite[Corollary 5]{BF} and \cite[Proposition 5]{BF}, respectively).

\item $\mu(t)$ may converge to a geometric structure on a different homogeneous space, as explained in Remark \ref{conv} (see \cite[Corollary 4]{BF}).  This occurs already in dimension $3$ for the Ricci flow (see \cite{homRF}).  

\item Many geometric flows have been studied in the homogeneous case using the bracket flow, including the Ricci flow \cite{Gzh, Pyn, GlcPyn, nilricciflow, homRF, Arr, Lfn, homRS}, the Chern-Ricci flow \cite{SCF, CRF}, the pluriclosed flow \cite{EnrFinVzz, ArrLfn}, the hermitian curvature flow \cite{LfnPujVzz}, the symplectic curvature flow \cite{SCFmuA} and the $G_2$-Laplacian flow \cite{FrnFinMnr, LF, BggFin, MrnSrp}.  

\item The bracket flow was recently applied by B\"ohm-Lafuente in \cite{BhmLfn,BhmLfn3} to prove the convergence of any immortal homogeneous Ricci flow solution to a homogeneous expanding Ricci soliton.  
\end{enumerate}

\subsection{Bracket flow evolution of solitons}  
In the light of the equivalence between the flows \eqref{flow-LG} and \eqref{BF}, a natural question arises: how do solitons evolve according to the bracket flow?  It is natural to expect an evolution of a very special kind.  

If $\mu$ is a fixed point (up to scaling) of the bracket flow \eqref{BF}, i.e.\ $\mu(t)=c(t)\cdot\mu$ for some $c(t)\in\RR^*$ (see \eqref{scmu}), then by just evaluating the equation at $t=0$ one obtains 
\begin{equation}\label{algsol}
Q_\mu=cI+D_\pg, \qquad \mbox{for some}\quad c\in\RR, \quad D=\left[\begin{matrix} 0&0\\ 0&D_\pg\end{matrix}\right]\in\Der(\ggo/\kg).
\end{equation} 
This implies that $(G_\mu/K_\mu,\gamma)$ is an algebraic soliton (see \eqref{sas-Q} and \eqref{as-Q}) satisfying \eqref{D00}.  The converse assertion also holds, see Corollary \ref{sasBF-cor} below.     

We now describe the bracket flow evolution of any semi-algebraic soliton.  

\begin{proposition}\label{sasBF}
Let $(G_\mu/K_\mu,\gamma)$ be a semi-algebraic soliton as in \eqref{semi-alg-sol}, say $q(\gamma)=c\gamma+\theta(D_\pg)$, and assume that \eqref{D00} holds.  Then the bracket flow solution starting at $\mu$ is given by  
\begin{equation}\label{saBF2}
\mu(t) = c(t)^{\frac{1}{s-r}} \cdot \left(\left[\begin{matrix} I&0\\ 0& e^{s(t)A} \end{matrix}\right] \cdot \mu\right), \qquad A := \proy_{\glg(\pg)_\gamma}(D_\pg),
\end{equation}
where $c(t)$ and $s(t)$ are as in \eqref{sas-evg} and $\proy_{\glg(\pg)_\gamma}:\glg(\pg)\rightarrow  \glg(\pg)_\gamma$ is the projection relative to the decomposition $\glg(\pg)=\glg(\pg)_\gamma\oplus\qg_\gamma$ given in \eqref{dec-g}.
\end{proposition}

\begin{proof}
It follows from \eqref{sas-evg} that $\gamma(t)=c(t) e^{s(t)D_\pg}\cdot\gamma$ and thus
\begin{align*}
q(\gamma(t)) =&  c(t)^\alpha e^{s(t)D_\pg}\cdot q(\gamma) = cc(t)^\alpha e^{s(t)D_\pg}\cdot\gamma + c(t)^\alpha \theta(D_\pg)e^{s(t)D_\pg}\cdot\gamma \\ 
=& cc(t)^{\alpha-1}\gamma(t)+c(t)^{\alpha-1} \theta(D_\pg)\gamma(t). 
\end{align*}
This implies that 
\begin{align}
Q_{\gamma(t)} =& \tfrac{cc(t)^{\alpha-1}}{s-r}I + c(t)^{\alpha-1}\proy_{\gamma(t)}(D_\pg) \notag \\
=& \tfrac{cc(t)^{\alpha-1}}{s-r}I + c(t)^{\alpha-1}e^{s(t)D_\pg}\proy_{\gamma}(D_\pg)e^{-s(t)D_\pg} \label{Qgt}\\
=& e^{s(t)D_\pg}\left(\tfrac{cc(t)^{\alpha-1}}{s-r}I + c(t)^{\alpha-1}(D_\pg-A)\right)e^{-s(t)D_\pg} \notag\\ 
=& c(t)^{\alpha-1}e^{s(t)D_\pg}Q_\gamma e^{-s(t)D_\pg}, \notag
\end{align}
where $A:=\proy_{\glg(\pg)_\gamma}(D_\pg)$.  According to Theorem \ref{BF-thm}, one can compute the bracket flow solution $\mu(t)$ by solving the differential equation given in its part (i).  It is straightforward to check that the solution is given by 
$$
h(t):= c(t)^{\frac{1}{r-s}}e^{s(t)A}e^{-s(t)D_\pg},
$$
where $c(t)$ and $s(t)$ are as in \eqref{sas-evg}, and therefore formula \eqref{saBF2} follows from \eqref{D00}, concluding the proof.  
\end{proof}

\begin{corollary}\label{sasBF-cor}
$(G_\mu/K_\mu,\gamma)$ is an algebraic soliton satisfying \eqref{D00} if and only if $\mu$ is a fixed point (up to scaling) of the bracket flow.   
\end{corollary}

In particular, algebraic solitons are precisely the possible limits of any normalized bracket flow (cf.\ Remark \ref{conv}).  Furthermore, given a starting point, one can obtain at most one non-flat algebraic soliton as a limit by running all possible normalized bracket flow solutions.  

We note that for expanding (resp.\ shrinking) solitons, the function $s$ is strictly increasing and $s(t)\to\infty$ as $t\to\infty$ (resp.\ $s(t)\to-\infty$ as $t\to-\infty$).  Recall also that $s(t)=t$ for steady solitons.  

In the case when $A$ is skew-symmetric (e.g.\ if a metric is attached to $\gamma$, since in that case $\glg(\pg)_\gamma\subset\sog(\pg)$), its eigenvalues are either purely imaginary numbers or zero, say $\pm\im a_1,\dots,\pm\im a_m,0,\dots,0$ ($a_j>0$).  Thus exactly one of the following two behaviors occur:

\begin{enumerate}[{\rm (i)}]
\item If the set $\{ a_1,\dots,a_m\}$ is linearly dependent over $\QQ$, then there exists a sequence $t_k$, with $t_k \rightarrow \pm\infty$ (depending on the type of soliton), such that $e^{s(t_k)A}=I$ for all $k$, and thus the bracket flow solution projected on the sphere is periodic.  

\item If on the contrary, the set $\{a_1,\dots,a_m\}$ is linearly independent over $\QQ$, then by Kronecker's theorem, for each $t_0\in(T_-,T_+)$, there exists a sequence $t_k$, with $t_k \rightarrow \pm\infty$, such that $e^{s(t_k)A} \rightarrow e^{s(t_0)A}$.  This implies that the solution projected on the sphere is not periodic but develop the following chaotic behavior: for each $t_0$ there exists a sequence $t_k\rightarrow\infty$ such that $\mu(t_k)/|\mu(t_k)|$ converges to $\mu(t_0)/|\mu(t_0)|$, that is, each point of the solution on the sphere is contained in the $\omega$-limit.  
\end{enumerate}

There are examples of $G_2$-Laplacian (see \cite{Ncl,LF}) and pluriclosed (see \cite{ArrLfn}) semi-algebraic solitons satisfying (i).  However, the existence of a semi-algebraic soliton as in part (ii) is still an open problem.

\section{Simultaneous diagonalization of a homogeneous geometric flow}\label{fd-sec}

In this section, we aim to characterize algebraic solitons among homogeneous solitons in a geometric way.  It will be assumed here that $\gamma$ comes with a Riemannian metric $g_\gamma$ (see Remark \ref{metg}).  

\begin{definition}\label{diag}
A homogeneous geometric structure $(M,\gamma)$ is said to be {\it flow diagonal} if the $\Aut(M,\gamma)$-invariant solution $\gamma(t)$ starting at $\gamma$ (see Proposition \ref{exist}) satisfies the following property: at some point $p\in M$, there exists an orthonormal basis $\beta$ of $T_pM$ such that the matrix $[Q_{\gamma(t)}]_\beta$ is diagonal for all $t$.
\end{definition}

Note that the point $p$ can be replaced by any other point of $M$ by homogeneity.  The property of being flow diagonal is a geometric invariant since given $f\in\Diff(M)$, the operators corresponding to $(M,f^*\gamma)$ are simultaneously conjugate, via $df|_{f^{-1}(p)}$, to those of $(M,\gamma)$.

If $(M,\gamma)$ is flow diagonal, then the velocity of the geometric flow $q(\gamma(t))=\theta(Q_{\gamma(t)})\gamma(t)$ takes a very simple form in terms of the fixed basis $\beta$, making the study of any aspect of the solutions to the ODE \eqref{flow-LG} much more workable, including its qualitative behavior and the search for exact solutions.  Indeed, once we consider a presentation $(M,\gamma)=(G/K,\gamma)$ and a reductive decomposition $\ggo=\kg\oplus\pg$, recall from Theorem \ref{BF-thm} that $\gamma(t)=h(t)^*\gamma$ for the solution $h(t)\in\Gl(\pg)$ to the ODE $\ddt h(t)=-h(t)Q_{\gamma(t)}$, $h(0)=I$.  It follows that the following conditions are equivalent:
\begin{enumerate}[{\rm (i)}]
\item $(G/K,\gamma)$ is flow diagonal.

\item The family of operators $\{ h(t):t\in (T_-,T_+)\}$ is simultaneously diagonalizable with respect to an orthonormal basis of $\pg$.
\end{enumerate}

For instance, the flow diagonal condition was assumed in the study of the Ricci flow on $4$-dimensional homogeneous manifolds in \cite{IsnJckLu,Ltt} and it was shown to hold in most of the exact Laplacian flow solutions on nilpotent Lie groups given in \cite[Section 4]{FrnFinMnr}.

\begin{example}
Let $(G/K,\gamma)$ be a simply connected semi-algebraic soliton such that \eqref{D00} holds.  It follows from \eqref{Qgt} that 
\begin{equation}\label{sas-Qt}
Q_{\gamma(t)} = c(t)^{\alpha-1}e^{s(t)D_\pg} Q_{\gamma} e^{-s(t)D_\pg}, \qquad\forall t.
\end{equation}
In particular, if $(G/K,\gamma)$ is actually an algebraic soliton, then $[Q_\gamma,D_\pg]=0$ and so $Q_{\gamma(t)}=c(t)^{\alpha-1}Q_\gamma$.  Thus $(G/K,\gamma)$ is flow diagonal as soon as only $Q_\gamma$ is diagonalizable with respect to an orthonormal basis, i.e.\ $Q_\gamma$ symmetric.  
\end{example}

We now show that the flow diagonal condition characterizes algebraic solitons among homogeneous solitons.  Recall that according to Theorem \ref{sas-hom}, any homogeneous soliton can be presented as a semi-algebraic soliton.  

\begin{theorem}\label{diagalg}
A semi-algebraic soliton $(G/K,\gamma)$ with $K$ compact and $Q_\gamma$ symmetric is flow diagonal if and only if it is an algebraic soliton.
\end{theorem}

\begin{remark}
In particular, any semi-algebraic soliton which is equivalent (via a diffeomorphism) to an algebraic soliton must be an algebraic soliton itself.
\end{remark}

\begin{proof}
We have that $Q_\gamma=cI+S$, where $D_\pg=A+S$, $A\in\glg(\pg)_\gamma$, $A^t=-A$ (since there is a metric attached to $\gamma$), and $S\in\qg_\gamma$, $S^t=S$ (since $Q_\gamma^t=Q_\gamma$).  According to Proposition \ref{Kcomp}, we can assume that $D$ satisfies \eqref{D00}.  If $(G/K,\gamma)$ is flow diagonal, then it follows from \eqref{sas-Qt} that 
$$
[e^{sD_\pg}Q_{\gamma}e^{-sD_{\pg}}, Q_{\gamma}] = 0,  \qquad \forall s\in(-\epsilon,\epsilon).
$$
Hence $[[D_\pg,Q_{\gamma}],Q_{\gamma}] = 0$ and 
$$
0=\tr{D_\pg[[D_\pg,Q_{\gamma}],Q_{\gamma}]}= -\tr{[D_\pg,Q_{\gamma}]^2} = -\tr{[A,S]^2},
$$
from which follows that $[A,S]=0$ as it is symmetric, and thus $D_\pg$ is normal.  This implies that $D$ is normal as well relative to any extension of the inner product to $\ggo$ and so $D^t\in \Der(\ggo)$. Thus $Q_\gamma = cI + \unm(D_\pg+D_\pg^t)$, with $\unm(D+D^t) \in \Der(\ggo)$,  showing that $(G/K,\gamma)$ is an algebraic soliton.
\end{proof}

\end{document}